\documentclass[a4paper,10pt]{amsart}

\usepackage{graphicx}
\usepackage{geometry}
\usepackage{amsthm}
\usepackage{amssymb}
\usepackage{amsmath}
\usepackage{hyperref}
\usepackage{xcolor}
\usepackage{enumerate}
\usepackage{listings}
\usepackage{complexity}
\usepackage{blkarray}
\usepackage{lmodern}

\usepackage[font=small,labelfont=bf]{caption}
\usepackage[font=small,labelfont=normalfont,labelformat=simple]{subcaption}
\graphicspath{{figs/}}

\newtheorem{theorem}{Theorem}

\newtheorem{lemma}[theorem]{Lemma}

\newtheorem{conjecture}{Conjecture}
\newtheorem{problem}{Problem}

\author
{
Raphael Steiner 
}
\thanks{Institute of Theoretical Computer Science, ETH Z\"{u}rich, Switzerland,  \texttt{raphaelmario.steiner@inf.ethz.ch}.
This work was supported by an ETH Postdoctoral Fellowship.}

\date{\today}

\title{Improved lower bound for the list chromatic number of graphs with no $K_t$ minor}

\begin{document}
\maketitle

\begin{abstract}
Hadwiger's conjecture asserts that every graph without a $K_t$-minor is $(t-1)$-colorable. It is known that the exact version of Hadwiger's conjecture does not extend to list coloring, but it has been conjectured by Kawarabayashi and Mohar (2007) that there exists a constant $c$ such that every graph with no $K_t$-minor has list chromatic number at most $ct$. More specifically, they also conjectured that this holds for $c=\frac{3}{2}$. 

Refuting the latter conjecture, we show that the maximum list chromatic number of graphs with no $K_t$-minor is at least $(2-o(1))t$, and hence $c \ge 2$ in the above conjecture is necessary. This improves the previous best lower bound by  Bar\'{a}t, Joret and Wood (2011), who proved that $c \ge \frac{4}{3}$. Our lower-bound examples are obtained via the probabilistic method.
\end{abstract}

\section{Introduction}
\paragraph{\textbf{Preliminaries.}} Given a number $t \in \mathbb{N}$, a \emph{$K_t$-minor} is a graph $G$ whose vertex-set can be partitioned into $t$ pairwise disjoint non-empty sets $Z_1,\ldots,Z_t$, such that for every $i \in [t]$, the induced subgraph $G[Z_i]$ is connected, and furthermore, for every two distinct $i, j \in [t]$, there exists at least one edge in $G$ with endpoints in the sets $Z_i$ and $Z_j$. We say that a graph \emph{contains $K_t$ as a minor} or that it \emph{contains a $K_t$-minor} if it admits a subgraph which is a $K_t$-minor. 

Given a graph $G$ and a color-set $S$, a \emph{proper coloring} of $G$ with colors from $S$ is a mapping $c:V(G) \rightarrow S$ such that $c^{-1}(s)$ is an independent set, for every $s \in S$. Given a graph $G$, a \emph{list assignment} for $G$ is an assignment $L:V(G) \rightarrow 2^\mathbb{N}$ of finite sets $L(v)$ (called lists) to the vertices $v \in V(G)$. An \emph{$L$-coloring} of $G$ is defined as a proper coloring $c:V(G) \rightarrow \mathbb{N}$ of $G$ in which every vertex is assigned a color from its respective list, i.e., $c(v) \in L(v)$ for every $v \in V(G)$. 
With this, we may define the chromatic number $\chi(G)$ of a graph $G$ as the smallest integer $k \ge 1$ such that $G$ admits an $L$-coloring, where $L(v):=[k]$ for every $v \in V(G)$. 

In a similar way, the \emph{list chromatic number} $\chi_\ell(G)$ of a graph $G$ is defined as the smallest number $k \ge 1$ such that $G$ admits an $L$-coloring for \emph{every} assignment $L(\cdot)$ of color lists to the vertices of $G$, provided that $|L(v)| \ge k$ for every $v \in V(G)$. 

Clearly, $\chi(G) \le \chi_\ell(G)$ for every graph $G$, but in general $\chi_\ell(G)$ is not bounded from above by a function in $\chi(G)$, as shown e.g. by complete bipartite graphs. 

\bigskip

Hadwiger's conjecture, arguably one of the most important open problems in graph theory, states the following upper bound on the chromatic number of graphs with no $K_t$-minor:
\begin{conjecture}[Hadwiger~\cite{hadwiger}, 1943]\label{hadwiger}
For every $t \in \mathbb{N}$, if $G$ is a graph not containing a $K_t$-minor, then $\chi(G) \le t-1$.
\end{conjecture}
Hadwiger's conjecture and its variants have received a lot of attention in the past, and a very good overview of partial results on this topic until about $2$ years ago can be found in the survey article~\cite{survey} by Seymour. In the following, let us briefly highlight the milestone results regarding Hadwiger's conjecture obtained so far. 

By a result of Wagner~\cite{wagner}, it was known that the case $t=5$ of Hadwiger's conjecture is equivalent to the statement of the famous \emph{four color conjecture}. After its confirmative resolution by Appel, Haken and Koch~\cite{appelhaken1,appelhaken2} in 1977, Hadwiger's conjecture had been proved for all values $t \le 5$. Notably, in 1993, Robertson, Seymour and Thomas~\cite{robertson} managed to go one step further and to prove Hadwiger's conjecture also for the case $t=6$. As of today, all the cases $t \ge 7$ of Hadwiger's conjecture remain open problems. 

The evident difficulty of the exact version of Hadwiger's conjecture has inspired many researchers to study its asymptotic relaxation, known as the \emph{Linear Hadwiger's conjecture}:
\begin{conjecture}
There exists an absolute constant $c>0$ such that every graph $G$ not containing a $K_t$-minor satisfies $\chi(G) \le ct$. 
\end{conjecture}
While also the Linear Hadwiger's conjecture remains open, there has been a lot of progress. By classical results of Kostochka~\cite{kostochka} and Thomason~\cite{thomason} from 1984, it was known that $K_t$-minor free graphs are $O(t\sqrt{\log t})$-colorable. While already quite close to a linear bound, it has proven difficult to overcome this order of magnitude during more than $30$ years of research. Finally, in 2019, Norine, Postle and Song~\cite{norine} managed to break this barrier, by proving that the maximum chromatic number of $K_t$-minor free graphs is in $O(t(\log t)^{1/4+o(1)})$. Very soon afterwards, several related results, extensions and improvements of this bound have been obtained, see~\cite{norine2, postle, postle2, postle3}. The current state of the art-bound of $O(t \log \log t)$ has been obtained only couple of months ago by Delcourt and Postle~\cite{del}. 

Parallel to the development of Hadwiger's conjecture, which concerns the ordinary chromatic number, the list chromatic number of graphs with no $K_t$-minor has also received a considerable amount of interest. 
For example, Borowiecki~\cite{borow} asked whether every graph with no $K_t$-minor has list-chromatic number at most $t-1$ (which would strengthen Hadwiger's conjecture). While this is easily seen to be true for $t \le 4$, already for $t=5$ there exist examples of planar graphs (hence $K_5$-minor free) with list chromatic number $5$, as constructed first by Voigt~\cite{voigt}. Later, Thomassen~\cite{thomassen} proved that $5$ is also the correct upper bound for the list chromatic number of planar graphs, and using Wagner's result~\cite{wagner}, this carries over to $K_5$-minor free graphs. For every $t \ge 6$, the maximum list chromatic number of $K_t$-minor free graphs remains unknown.

Since the exact version of Hadwiger's conjecture does not extend to list coloring, it is natural to study asymptotic versions. 
The current state of the art-upper bound on the list chromatic number of $K_t$-minor free graphs is $O(t(\log \log t)^2)$, as was recently proved by Delcourt and Postle~\cite{del}. Compare also~\cite{norine2,postle3} for the previous asymptotic upper bounds of magnitudes $O(t(\log t)^{1/4+o(1)})$ and $O(t(\log \log t)^6)$, respectively.

The following \emph{List Hadwiger conjecture} was first stated by Kawarabayashi and Mohar~\cite{kawarabayashi} in 2007, compare also the entry~\cite{op} in the Open Problem Garden.  

\begin{conjecture}\label{listhadwiger}
There exists an absolute constant $c>0$ such that every $K_t$-minor free graph $G$ satisfies $\chi_\ell(G) \le ct$. 
\end{conjecture}
At first, an even stronger conclusion, namely that every $K_t$-minor free graph is $t$-list-colorable, was believed to be possible, compare e.g.~\cite{kawarabayashi,wood}. However, this stronger conjecture was disproved by the following result of Bar\'{a}t, Joret and Wood~\cite{barat} from 2011, which shows that $c \ge \frac{4}{3}$ in Conjecture~\ref{listhadwiger} is necessary.
\begin{theorem}
For every integer $t \ge 1$ there exists a graph with no $K_{3t+2}$-minor and list chromatic number greater than $4t$. 
\end{theorem}

Kawarabayashi and Mohar stated in~\cite{kawarabayashi} that they believe that Conjecture~\ref{listhadwiger} holds true for $c=\frac{3}{2}$, and this statement also appears as Conjecture~8.4 in the survey article~\cite{survey} by Seymour:
\begin{conjecture}\label{3halfcon}
Every graph $G$ without a $K_t$-minor satisfies $\chi_\ell(G) \le \frac{3}{2}t$.
\end{conjecture}

In this note, we disprove Conjecture~\ref{3halfcon} by showing that the maximum list chromatic number of $K_t$-minor free graphs is at least $2t-o(t)$, and hence $c \ge 2$ in Conjecture~\ref{listhadwiger} is necessary. The proof enhances an idea from~\cite{barat} by using probabilistic arguments.

\begin{theorem}\label{main}
For every $\varepsilon \in (0,1)$ there is $t_0=t_0(\varepsilon)$ such that for every $t \ge t_0$ there exists a graph with no $K_t$-minor and list chromatic number at least $(2-\varepsilon)t$. 
\end{theorem}

It would be interesting to see whether our lower-bound construction could be optimal up to the lower-order term, or whether further improvement of the lower bound is possible. 
\begin{problem}
Does every $K_t$-minor-free graph $G$ satisfy $\chi_\ell(G) \le 2t$?
\end{problem}

\section{Proof of Theorem~\ref{main}}

In the following, for a natural number $n \in \mathbb{N}$ and a probability $p \in [0,1]$, we denote by $G(n,n,p)$ the bipartite Erd\H{o}s-Renyi graph, that is, a random bipartite graph $G$ with bipartition $A \cup B$ such that $|A|=|B|=n$, in which every pair $ab$ with $a \in A, b \in B$ is selected as an edge of $G$ with probability $p$, independently from all other such pairs. 

\begin{lemma}\label{random}
Let $\varepsilon \in (0,1)$ be fixed, let $f=f(\varepsilon) \in \mathbb{N}$ and $\delta=\delta(\varepsilon) \in (0,1)$ be constants chosen such that $f \delta<1$. Let $p=p(n):=n^{-\delta}$. Then w.h.p. as $n \rightarrow \infty$, the random graph $G=G(n,n,p(n))$ with bipartition $A \cup B$ satisfies both of the following properties:
\begin{itemize}
\item For every subset $X \subseteq A$ such that $|X| \ge \varepsilon n$ and every collection of pairwise disjoint non-empty subsets $Y_1,\ldots,Y_k \subseteq B$ such that $k \ge \varepsilon n$ and $\max\{|Y_1|,\ldots,|Y_k|\} \le f$, there exists a vertex $x \in X$ and some $j \in [k]$ such that $G$ contains all the edges $xy, y \in Y_j$. The same statement holds symmetrically for the case when $X \subseteq B$ and $Y_1,\ldots,Y_k \subseteq A$.
\item $G$ has maximum degree at most $\varepsilon n$. 
\end{itemize}
\end{lemma}
\begin{proof}
\noindent
\begin{itemize}
\item Let $E_n$ denote the probability event that $G$ does not satisfy the property claimed in the first item. We need to show that $\mathbb{P}(E_n) \rightarrow 0$ as $n \rightarrow \infty$. So consider a fixed subset $X \subseteq A$ (or symmetrically, $X \subseteq B$) such that $|X| \ge \varepsilon n$, and a fixed collection $Y_1,\ldots,Y_k$ of disjoint non-empty subsets of $B$ (or symmetrically, $A$), where $k \ge \varepsilon n$ and $\max\{|Y_1|,\ldots,|Y_k|\} \le f$. Let $E(X,Y_1,\ldots,Y_k)$ be the probability event ``there exists no pair $(x,j) \in X \times [k]$ such that $x$ is fully connected to the vertices in $Y_j$''. Fixing a vertex $x \in X$ and an index $j \in [k]$, clearly the probability of the event that ``$x$ is not fully connected to $Y_j$'' equals $1-p^{|Y_j|} \le 1-p^{f}$.  Since these events are independent for different choices of $(x,j)$, we conclude that 
$$\mathbb{P}(E(X,Y_1,\ldots,Y_k)) \le (1-p^{f})^{|X|\cdot k} \le (1-p^{f})^{\varepsilon^2n^2}\le \exp(-p^{f}\varepsilon^2 n^2)=\exp(-\varepsilon^2n^{2-f\delta}).$$
With a rough estimate, there are at most $$2 \cdot 2^n\cdot (n+1)^{n} \le \exp(\ln(2)(n+1)+n \ln (n+1))$$ ways to select the sets $X,Y_1,\ldots,Y_k$. Hence, applying a union bound we find that 
$$\mathbb{P}(E_n) \le \exp(\ln(2)(n+1)+n \ln (n+1)-\varepsilon^2n^{2-f\delta}).$$
The right hand side of the above inequality tends to $0$ as $n \rightarrow \infty$, since $f\delta<1$ and hence $\varepsilon^2n^{2-f\delta}=\Omega(n^{2-f\delta})$ grows faster than $\ln(2)(n+1)+n \ln (n+1)=O(n\ln n)$. This proves that $G$ satisfies the properties claimed by the first item w.h.p., as required.
\item To show that also the property claimed by the second item holds true w.h.p., consider the probability that a fixed vertex $x \in A \cup B$ has more than $\varepsilon n$ neighbors in $G$. Note that the degree of $x$ in $G(n,n,p)$ is distributed like a binomial random variable $B(n,p)$, and hence its expectancy is $np=n^{1-\delta}$, which is smaller than $\frac{\varepsilon n}{2}$ for $n$ sufficiently large in terms of $\varepsilon$ and $\delta$. Applying Chernoff's bound we find for every sufficiently large $n$:
$$\mathbb{P}(d_G(x)>\varepsilon n )\le \mathbb{P}(B(n,p)>2np) \le \exp\left(-\frac{1}{3}np\right)=\exp\left(-\frac{1}{3}n^{1-\delta}\right).$$
Since this bound applies to every choice of $x \in A \cup B$, applying a union bound we find that the probability that $G$ has maximum degree more than $\varepsilon n$ is at most
$$2n\exp\left(-\frac{1}{3}n^{1-\delta}\right)=\exp\left(\ln(2n)-\frac{1}{3}n^{1-\delta}\right)$$ which tends to $0$ as $n \rightarrow \infty$, as desired (here we used that $\delta<1$ and hence $n^{1-\delta}$ grows faster than $\ln(2n)$). 
\end{itemize}

\end{proof}

The next lemma uses Lemma~\ref{random} to obtain a useful deterministic statement about the existence of graphs with certain properties, which are then handy when constructing the lower-bound examples for Theorem~\ref{main}.

\begin{lemma}\label{cor}
For every $\varepsilon \in (0,1)$, there is $n_0=n_0(\varepsilon)$ such that for every $n \ge n_0$, there exists a graph $H$ whose vertex-set $V(H)=A \cup B$ is partitioned into two disjoint sets $A$ and  $B$ of size $n$, and such that the following properties hold:
\begin{itemize}
\item Both $A$ and $B$ are cliques of $H$,
\item every vertex in $H$ has at most $\varepsilon n$ non-neighbors in $H$, and
\item for every $t \in \mathbb{N}$ such that $t \ge (1+2\varepsilon) n$, $H$ does not contain $K_t$ as a minor. 
\end{itemize}
\end{lemma}
\begin{proof}
Let $f:=\lceil \frac{1}{\varepsilon} \rceil \in \mathbb{N}$ and $\delta:=\frac{\varepsilon}{2}$. Then $f\delta <1$, and hence we may apply Lemma~\ref{random}. It follows directly that there exists $n_0=n_0(\varepsilon) \in \mathbb{N}$ such that for every $n \ge n_0$ there exists a bipartite graph $G$, whose bipartition classes $A$ and $B$ are both of size $n$, and such that 
\begin{itemize}
\item For every subset $X \subseteq A$ such that $|X| \ge \varepsilon n$ and every collection of pairwise disjoint non-empty subsets $Y_1,\ldots,Y_k \subseteq B$ such that $k \ge \varepsilon n$ and $\max\{|Y_1|,\ldots,|Y_k|\} \le f$, there exists a vertex $x \in X$ and some $j \in [k]$ such that $G$ contains all the edges $xy, y \in Y_j$. The same statement holds symmetrically for the case when $X \subseteq B$ and $Y_1,\ldots,Y_k \subseteq A$.
\item $G$ has maximum degree at most $\varepsilon n$. 
\end{itemize}
We now define $H$ as the complement of $G$ (also with vertex-set $A \cup B$). Since $G$ is bipartite, clearly $A$ and $B$ form cliques in $H$, verifying the first item. The second item follows directly from the fact that $\Delta(G) \le \varepsilon n$.

It hence remains to verify the last item. Towards a contradiction, suppose that there exists a number $t \in \mathbb{N}$, $t \ge (1+2\varepsilon)n$, such that $H$ contains $K_t$ as a minor. This implies that there exists a collection $\mathcal{Z}$ of non-empty and pairwise disjoint subsets of $V(H)$ such that $|\mathcal{Z}|=t$ and such that for every two distinct $Z, Z' \in \mathcal{Z}$, there exists at least one edge in $H$ connecting a vertex in $Z$ to a vertex in $Z'$. Let us denote $\mathcal{Z}_s:=\{Z \in \mathcal{Z}||Z|=s\}$, and $z_s:=|\mathcal{Z}_s|$, for every $s \ge 1$. We clearly have
$$2n \ge \sum_{s\ge 1}{sz_s} \ge 2(t-z_1)+z_1=2t-z_1 \ge 2n+4\varepsilon n -z_1.$$
Rearranging yields that $z_1 \ge 4\varepsilon n$. From this we may conclude that either $A$ or $B$ contains at least $2 \varepsilon n$ singletons from $\mathcal{Z}$. By symmetry (possibly by renaming $A$ and $B$), we may assume w.l.o.g. that $B$ contains at least $2 \varepsilon n$ singletons from $\mathcal{Z}$, and denote the set of these singletons by $X$. Let us now define $\mathcal{Z}_A:=\{Z \in \mathcal{Z}|Z \subseteq A\}$ and $\mathcal{Z}_B:=\{Z \in \mathcal{Z}|Z \cap B \neq \emptyset\}$. Since the sets in $\mathcal{Z}$ are pairwise disjoint, we can see that $|\mathcal{Z}_B| \le |B|=n$, and therefore $|\mathcal{Z}_A|=t-|\mathcal{Z}_B| \ge t-n \ge 2\varepsilon n$. Since $|A|=n$, the latter implies that there are at least $\varepsilon n$ distinct sets in $\mathcal{Z}_A$ which have size at most $\frac{1}{\varepsilon}$. Let $Y_1,\ldots,Y_k$ with $k \ge \varepsilon n$ be an enumeration of the sets in $\mathcal{Z}_A$ of size at most $\frac{1}{\varepsilon} \le f$. By the above, there exists $j \in [k]$ and a vertex $x \in X$ such that $xy \in E(G)$ for every $y \in Y_j$. Since $H$ is the complement of $G$, this means that $\{x\}$ and $Y_j$ are distinct sets in $\mathcal{Z}$, which do not have any connecting edge. This is a contradiction to our initial assumptions, and hence we have shown that the third item claimed in the lemma is also satisfied. This concludes the proof. 
\end{proof}

We are now ready for the proof of Theorem~\ref{main}. The only missing ingredient is the following well-known ``pasting-lemma'', compare Lemma~3 in~\cite{barat}. 

\begin{lemma}\label{glue}
Let $G_1$ and $G_2$ be $K_t$-minor free graphs, and let $V(G_1) \cap V(G_2)=C$. If $C$ forms a clique in both $G_1$ and $G_2$, then the graph $G_1 \cup G_2$ is also $K_t$-minor free.  
\end{lemma}

\begin{proof}[Proof of Theorem~\ref{main}]
Let a fixed $\varepsilon \in (0,1)$ be given. Pick some $\varepsilon' \in (0,1)$ such that $\frac{2-\varepsilon'}{1+2\varepsilon'} \ge 2-\frac{\varepsilon}{2}$.  Let $n_0=n_0(\varepsilon') \in \mathbb{N}$ be as in Lemma~\ref{cor}, and define $t_0:=\max\{\lceil(1+2\varepsilon')n_0\rceil,\left\lfloor \frac{4}{\varepsilon}\right\rfloor\}$. Now, let $t \ge t_0$ be any given integer. Define $n:=\left\lfloor\frac{t}{1+2\varepsilon'}\right\rfloor \ge n_0$. Applying Lemma~\ref{cor}, we find that there exists a graph $H$ whose vertex-set is partitioned into two non-empty sets $A$ and $B$ of size $n$, such that both $A$ and $B$ form cliques in $H$, every vertex in $H$ has at most $\varepsilon' n$ non-neighbors, and $H$ is $K_t$-minor free (since $t \ge (1+2\varepsilon')n$, by definition of $n$). 

For every possible assignment $c \in [2n-1]^A$ of colors from $[2n-1]$ to vertices in $A$, denote by $H(c)$ an isomorphic copy of $H$, such that the vertex-set of $H(c)$ decomposes into the sets $A$ and $B(c)$ of size $n$. More precisely, the distinct copies $H(c), c \in [2n-1]^A$ of $H$ share the same set $A$ but have pairwise disjoint sets $B(c)$. Since $A$ forms a clique of size $n$ in the $K_t$-minor-free graph $H(c)$ for every coloring $c:A \rightarrow [2n-1]$, it follows by repeated application of Lemma~\ref{glue} that the graph $\mathbf{G}$ with vertex set $A \cup \bigcup_{c \in [2n-1]^A}{B(c)}$, obtained as the union of the graphs $H(c), c \in [2n-1]^A$, is $K_t$-minor free. 

Now, consider an assignment $L:V(\mathbf{G}) \rightarrow 2^\mathbb{N}$ of color lists to the vertices of $\mathbf{G}$ as follows:
For every vertex $a \in A$, we define $L(a):=[2n-1]$, and for every vertex $b \in B(c)$ for some coloring $c \in [2n-1]^A$ of $A$, we define $L(b):= [2n-1] \setminus \{c(a)|a \in A, ab \notin E(H(c))\}$.
Note that since every vertex in $B(c)$ has at most $\varepsilon' n$ non-neighbors in $H(c)$, we have $|L(v)| \ge 2n-1-\varepsilon' n$ for every vertex $v \in V(\mathbf{G})$. 

We now claim that $\mathbf{G}$ does not admit a proper coloring with colors chosen from the lists $L(v), v \in V(\mathbf{G})$, which will then prove that $\chi_\ell(\mathbf{G}) \ge 2n-\varepsilon' n$. Indeed, suppose towards a contradiction there exists a proper coloring $c_\mathbf{G}:V(\mathbf{G}) \rightarrow \mathbb{N}$ of $\mathbf{G}$ such that $c_\mathbf{G}(v) \in L(v)$ for every $v \in V(\mathbf{G})$. Let $c$ be the restriction of $c_\mathbf{G}$ to $A$, and consider the proper coloring of $H(c)$ obtained by restricting $c_\mathbf{G}$. Since $v(H(c))=2n$ and $c_\mathbf{G}(v) \in [2n-1]$ for every $v \in V(H(c))$, there must exist two (necessarily non-adjacent) vertices in $H(c)$ which have the same color with respect to $c_\mathbf{G}$. Concretely, there exist $a \in A$, $b\in B(c)$ such that $ab \notin E(H(c))$ and $c_\mathbf{G}(a)=c_\mathbf{G}(b)$. This however yields a contradiction, since $c_\mathbf{G}(b) \in L(b)$ and by definition $c(a)=c_\mathbf{G}(a)$ is not included in the list of $b$. 

We conclude that indeed, $\mathbf{G}$ is a $K_t$-minor free graph which satisfies $$\chi_\ell(\mathbf{G}) \ge (2-\varepsilon')n=(2-\varepsilon')\left\lfloor\frac{t}{1+2\varepsilon'}\right\rfloor >(2-\varepsilon')\left(\frac{t}{1+2\varepsilon'}-1\right) \ge \left(2-\frac{\varepsilon}{2}\right)t-(2-\varepsilon') \ge (2-\varepsilon)t,$$ where for the last inequality we used $t \ge t_0 \ge \frac{4}{\varepsilon}$. 
\end{proof}

\end{document}